\pdfoutput=1  
\documentclass[11pt,a4paper]{amsart}
\usepackage[utf8]{inputenc}
\usepackage[T1]{fontenc}
\usepackage{colortbl}
\newcolumntype{C}{>{\columncolor[gray]{0.8}}c}
\usepackage{amsmath,amsthm,amssymb,amsrefs}
\usepackage{float}
\usepackage[top=1.2in, bottom=0.9in, left=1.3in, right=1.3in,marginpar=1in]{geometry}
\usepackage{etoolbox}
\usepackage{comment}

\usepackage{tikz}
\usetikzlibrary{decorations.markings,matrix,fadings,decorations.pathmorphing,calc,intersections,patterns}
\tikzset{
  between/.style args={#1 and #2}{
    at = ($(#1)!0.5!(#2)$)
  }
}
\tikzset{test/.style n args={3}{
    postaction={
      decorate,
      decoration={
        markings,
        mark=between positions 0 and \pgfdecoratedpathlength step 0.5pt with {
          \pgfmathsetmacro\myval{multiply(
            divide(
            \pgfkeysvalueof{/pgf/decoration/mark info/distance from start}, \pgfdecoratedpathlength
            ),
            100
            )};
          \pgfsetfillcolor{#3!\myval!#2};
          \pgfpathcircle{\pgfpointorigin}{#1};
          \pgfusepath{fill};}
      }}}}

\docsvlist{N,Z,Q,C,R,F,A,W}

\docsvlist{A,B,C,D,E,F,G,H,I,J,K,L,M,N,O,P,Q,R,S,T,U,V,W,X,Y,Z}

\docsvlist{A,B,C,D,E,F,G,H,I,J,K,L,M,N,O,P,Q,R,S,T,U,V,W,X,Y,Z,a,b,c,d,e,f,h,i,j,k,m,n,o,p,q,r,s,t,u,v,w,x,y,z}

\docsvlist{a,b,r,s,t,u,v,w,x,y,z}

\docsvlist{a,b,c,d,e,g,h,i,j,k,l,m,n,o,p,q,r,s,t,u,v,w,x,y,z,A,B,C,D,E,F,G,H,I,J,K,L,M,N,O,P,Q,R,S,T,U,V,W,X,Y,Z}

\docsvlist{Aff,ap,Ass,Aut,Bdry,Coar,cl,codim,coker,colim,cone,conv,diag,diam,dist,End,Ext,flat,lin,Lip,Hom,Id,im,Int,pillow,proj,rank,ray,rb,ri,pump,Simp,Span,spt,ssc,supp,Tan,Tot,vh,Vert,Vol,vol}

\DeclareMathOperator\cfk{CFK}

\def\wt#1{\widetilde{#1}}

\newtheorem{theorem}[equation]{Theorem}
\newtheorem{proposition}[equation]{Proposition}

\newtheorem{corollary}[equation]{Corollary}
\newtheorem{lemma}[equation]{Lemma}

\theoremstyle{definition}
\newtheorem{definition}[equation]{Definition}

\theoremstyle{remark}
\newtheorem{remark}[equation]{Remark}

\newtheorem*{ack}{Acknowledgments}
\numberwithin{equation}{section}

\title{Hyperbolic L-space knots not concordant to algebraic knots}
\author{Maciej Borodzik}
\address{Institute of Mathematics of Polish Academy of Science, ul \'Sniadeckich 8, 00-656 Warsaw, Poland}
\email{mcboro@mimuw.edu.pl}

\author{Masakazu Teragaito}
\address{International Institute for Sustainability with Knotted Chiral Meta Matter (WPI-SKCM$^2$), Hiroshima University, 
1-3-1 Kagamiyama, Higashi-Hiroshima, 739--8526, Japan}
\email{teragai@hiroshima-u.ac.jp}

\date{\today}

\begin{document}
\maketitle

\begin{abstract}
  We construct hyperbolic L-space knots that are not concordant to any linear combination of algebraic knots.
\end{abstract}
\section{Overview}
Describing knot concordance groups and understanding $\Z$-homology cobordism of three-manifolds are listed among the most important problems in low-dimensional topology. 
Among many questions, a particular interest is about the position of various classes of knots, respectively 3-manifolds, in the
concordance group, respectively in the group of $\Z$-homology cobordisms.

It is for example well-known, \cite{Myers}, that each 3-manifold is $\Z$-homology cobordant to a hyperbolic $3$-manifold. On the other hand,
there are 3-manifolds that are not $\Z$-homology cobordant to graph manifolds, \cite{CochranTanner}; recently \cite{DaiStoffregen} proved
that the graph manifolds span a subgroup of the group of $\Z$-homology cobordism with $\Z^\infty$-summand as a quotient.

For link cobordisms, there is an abundance of similar questions. There are knots that are not topologically concordant to alternating knots
\cite{FLZ}. A refinement of the argument shows that there exist knots that are not topologically concordant to L-space knots \cite{Ramazan}.
On the other hand, all knots are topologically concordant to strongly quasipositive knots \cite{BorodzikFeller}, a statement that is definitely false in the smooth category, because for all strongly quasipositive knots, all slice torus invariants are equal, compare \cite{FLL}.

An algebraic link is defined as a link of a plane curve singularity. All such links are graph links \cite{Eisenbud}. Also, all such
links are L-space links by \cite{Hedden,GorskyNemethi}.
Studying the position of algebraic knots in the whole knot concordance group seems to bring an immediate answer: while it is not stated explicitly in \cite{Ramazan}, the methods in that paper suggest that the quotient of the topological concordance group by the group of L-space knots has an infinite $\Z^\infty$ summand. However, the position of algebraic knots in the concordance subgroup spanned by all L-space knots seems rather misterious. This question seems to be a natural generalization of a question on the position of graph manifolds in the $\Z$-homology cobordism group
of $3$-manifolds.


The main result of the paper is the following.
\begin{theorem}\label{thm:main}
  There exist  infinitely many hyperbolic L-space knots that are not smoothly concordant to any linear combination of algebraic knots.
\end{theorem}
\begin{proof}[Plan of the proof]
In Section~\ref{sec:kn} we construct an infinite family of knots $K_n$. By Montesinos trick we show that each $K_n$ admits a positive L-space surgery.
The braid index of $K_n$ is bounded by $4$ by construction. Combining this fact with the lack of semigroup property of $K_n$, we show that $K_n$
are hyperbolic.

The proof that $K_n$ is not concordant to any linear combination of algebraic knots involves computing the function $\Upsilon$
of 
Ozsv\'ath, Stipsicz and Szab\'o, \cite{Upsilon}. By \cite{Upsilon}, this function is determined from the Alexander polynomial of $K_n$, and to
pass from that polynomial to $\Upsilon$, we use an explicit algorithm described in \cite{BorodzikHedden}. The computation of the Alexander polynomial of $K_n$, conducted in Section~\ref{sec:kn} involves a surgery description of $K_n$, and Torres formula.

The main obstruction is the following. It was shown by \cite{Tange}
that if $K$ is an algebraic knot, then $-3\int_0^2\Upsilon_K(t)\,dt$ is an integer. In fact, echoing the calculation of \cite{Borodzik_abelian},
one can express this integral via invariants of the underlying plane curves singularity, compare Section~\ref{sec:algebraic}.

On the contrary, Theorem~\ref{thm:integral_n} shows that $-3\int_0^2\Upsilon_{K_n}$ is a non-integral fraction of $5$. Therefore,
none of the knots $K_n$ can be expressed as linear combinations of algebraic knots.
\end{proof}
By using the table of \cite{BakerKegel}, 
we have found several L-space knots for which
$-3$ times the integral of $\Upsilon$ is not an integer. These are gathered in a table in Section~\ref{sec:specific}.

The methods do not show whether the knots $K_n$ generate a $\Z^\infty$ summand in the quotient subgroup
generated by the concordance classes of all L-space knots modulo algebraic knots. However, we know that a subsequence of $K_n$
knots is linearly independent in the topological concordance group, see Theorem~\ref{thm:independence}.

The structure of the paper is the following. Section~\ref{sec:review} gives a necessary background on the $\Upsilon$ function.
In Section~\ref{sec:algebraic}, we recall Tange's calculations of the $\Upsilon$ function and show some properties of the integral.
In particular, we introduce a purely Floer-theoretic invariant $\omega$ and relate it to invariants of plane curve singularities. These results are of independent interest, especially that they give an algebro-geometric motivation for studying $-3\int\Upsilon$ as a knot invariant.

The family of knots $K_n$ is constructed in Section~\ref{sec:kn}. Their main properties are stated in Theorem~\ref{thm:integral_n},
whose part is proved in that section. The most difficult part, showing that $K_n$ are indeed L-space knots, is proved in
Section~\ref{sec:montesinos}.

Section~\ref{sec:linear} addresses the question of linear independence of knots $K_n$. We study roots of the Alexander polynomial of
$K_n$ and show that a subsequence of $K_n$ is linearly independent. We also show that $\Upsilon$ function alone cannot prove
that $K_n$ are independent modulo the group of algebraic knots.

Finally, in Section~\ref{sec:specific} we provide a table
of those L-space knots of \cite{BakerKegel} for which $-3\int\Upsilon$ is not integral.

\begin{ack}
  The authors are grateful to Peter Feller for fruitful conversations.
  MB was supported by NCN OPUS 2019/B/35/ST1/01120 grant.
  MT has been supported by JSPS KAKENHI Grant Number 20K03587. 
\end{ack}

\section{Review of the $\Upsilon$ function}\label{sec:review}
Recall that to a knot $K$ in the $3$-sphere, knot Floer homology associates a complex $\cfk^\infty(K)$ over the ring $\Z_2[U,U^{-1}]$, where
$U$ is a formal variable. The complex is $\Z\oplus\Z$--filtered, $\Z$-graded, and the multiplication by $U$ lowers the grading by $2$
and the filtration by $(1,1)$. The complex is defined up to bifiltered chain homotopy equivalence.

In \cite{Upsilon}, a concordance invariant $\Upsilon$ was extracted from this chain complex, see also \cite{Livingston}. In short, for each $t\in[0,2]$,
one associates a collapsed filtration. If $x\in\cfk^\infty$ is at bifiltration level $(a,b)$, we define its
$\cF_t$-filtration level by $\frac{t}{2}a+(1-\frac{t}{2})b$. Denote by $\cC_{s,t}$ to be the subcomplex of $\cfk$ of elements at $\cF_t$-filtration level $\le s$. As $\cC_{s,t}$ is a subcomplex of $\cfk$, there is a map $H_i(\cC_{s,t})\to H_i(\cfk)$, where the subscript $i$
denotes the homological grading (the $\Z$-grading) of the complex $\cfk$. Set:
\[\nu(t)=\min\{s\colon H_0(\cC_{s,t})\to H_0(\cfk)\textrm{ is surjective}\}.\]
The function $\Upsilon$ is defined as
\[\Upsilon(t)=-2\nu(t).\]
Among many properties of the function $\Upsilon$, the most important for the sake of this paper is that it is a concordance invariant. That is to say, if $K_1$ is smoothly concordant to $K_2$, then $\Upsilon_{K_1}(t)=\Upsilon_{K_2}(t)$ for all $t\in[0,2]$.

If $K$ is an L-space knot, the complex $\cfk$ is determined by the Alexander polynomial, via so-called \emph{staircase complex}. The $\Upsilon$ function for such knots was computed in \cite{Upsilon}. 
To be more precise, write the Alexander polynomial $\Delta_K=1+(t-1)(t^{c_1}+\dots+t^{c_\ell})$ with $1\le c_1<\dots<c_{\ell}$ (such presentation is possible for all L-space knots, see \cite{OSlens}). Let $S_K=\Z_{\ge 0}\setminus\{c_1,\dots,c_\ell\}$.
\begin{definition}[see \cite{Wang}]\label{def:formal}
  The set $S_K$ is called the \emph{formal semigroup} of the L-space knot $K$.
\end{definition}
Another definition is that
\[
\frac{\Delta_K(t)}{1-t}=\sum_{s\in S_K} t^s.
\]
If $K$ is algebraic, the set $S_K$ is the semigroup of the underlying singularity. It is an interesting question to study
non-algebraic L-space knots for which $S_K$ has the structure of a semigroup \cite{Teragaito,Wang}. 

The $\Upsilon$ function of an L-space knots is related to the formal semigroup via the Fenchel--Legendre transform of an extension of the function $I\colon\Z\to\Z$, $I\colon n\mapsto \#\{S_K\cap[0,n)\}$; see \cite{BorodzikHedden}. In particular, for an L-space knot, $\Upsilon$ is convex.
It is interesting to see for which knots the $\Upsilon$ function is convex \cite{Himeno}.

\section{The integral of $\Upsilon$ in algebraic geometry}\label{sec:algebraic}

While it was proved by Tange \cite{Tange} that $-3\int\Upsilon$ for an algebraic knot is an integer, we recall his computations. We also show that
the quantity $-2\tau-3\int\Upsilon$ has a special interpretation in singularity theory. This explains our initial interest in the invariant $-3\int\Upsilon$.

We begin by recalling a standard definition, we refer to \cite{Brieskorn,Eisenbud}
for more details.

Let $z\in\C^2$ be a singular point of a complex algebraic curve $C$. Recall that the \emph{multiplicity} of a singular point, denoted
$m$, is the minimal positive number that can be obtained as a local intersection of $C$ at $z$ with another algebraic curve.
The fact that the point $z$ is singular means that $m>1$.

Suppose $z$ has a single branch. Set $m_1=m$. Blow up a singular point to obtain a new curve $\wt{C}$ and denote by $E$ the exceptional divisor
of the blow-up. Usually $\wt{C}$ will have a singular point at $\wt{C}\cap E$. Denote by $m_2$ its multiplicity. If $\wt{C}$ is
smooth, then $m_2=1$ and we stop the procedure. The procedure can be iterated until some strict transform $\wt{C}$ is smooth.
Eventually we obtain a finite sequence $(m_1,m_2,m_3,\ldots,m_n)$ of integers such that $m_i>1$ (usually we discard the last $1$ from the
sequence).

\begin{definition}
The sequence $(m_1,\ldots,m_n)$ is called the \emph{multiplicity sequence} of a unibranched singular point.
\end{definition}

The multiplicity sequence is a complete topological invariant of a singular point, in the sense that any two unibranched singular points with
the same multiplicity sequences are topologically equivalent. All topological invariants of the singularity can be computed from
the multiplicity sequence. For example, the following formula was proved by Milnor \cite{Milnor}, compare
\cite{Brieskorn,OZ}.
\begin{equation}\label{eq:milnor}
\mu_z=\sum_{i=1}^n m_i(m_i-1),
\end{equation}
where $\mu_z$ is the Milnor number of the critical point. The Milnor number is equal to twice the genus of the link of the singular point.

We introduce the following quantity.
\begin{equation}\label{eq:omega}
  \omega(K)=-3\int_0^2\Upsilon_K(t)dt-2\tau(K).
\end{equation}
As we see, $\omega$ is defined purely from Heegaard Floer-type invariants.

\begin{theorem}\label{thm:ouretasagree}
  Let $z$ be a singular point with one branch. Let $K$ be the link of singularity. Then $\omega(K)=\sum (m_i-1)$.
\end{theorem}
\begin{proof}
  The proof relies on the following result of Tange \cite{Tange}, based on \cite{FellerKrcatovich} and \cite{BodnarNemethi}.

\begin{lemma}\label{prop:peterdavid}
For $m>1$ define the function
\[\Upsilon_m(t)=-i(i+1) -\frac{1}{2}m(m-1-2i)t\textrm{ if }t\in\left[\frac{2i}{m},\frac{2i+2}{m}\right].\]
If a singular point $z$ has link $K$ and multiplicity sequence $(m_1,\ldots,m_n)$, then
\[\Upsilon_K=\sum_{i=1}^n \Upsilon_{m_i}.\]
\end{lemma}
Continuing the proof of Theorem~\ref{thm:ouretasagree},
we use the following formula of \cite{Tange}.

\begin{equation}\label{eq:integratem}
  \int_0^2\Upsilon_m(t)dt=
  \frac{-m^2+1}{3}.
\end{equation}
From \eqref{eq:integratem} and Lemma~\ref{prop:peterdavid} we conclude that
\[-3\int_0^2\Upsilon_K(t)=\sum (m_i^2-1).\]
For algebraic knots, $2\tau(K)=\mu_z$ is the Milnor number of the underlying singular point. We conclude by \eqref{eq:milnor}.
\end{proof}
\begin{corollary}[see \cite{Tange}]
  For an algebraic knot $-3\int_0^2\Upsilon(t)dt$ is an integer.
\end{corollary}

We have the following interpretation of $\omega(K)$, which is due to Orevkov and Zajdenberg \cite{OZ}.
Blow up the critical point $z$ until the reduced inverse image $D=\pi^{-1}(C)_{red}$ is a normal crossing divisor. Let $E_1,\ldots,E_s$
be the exceptional divisors. Define the canonical divisor $K_z=\sum \alpha_iE_i$ by the condition that $K_z\cdot E_i+E_i\cdot E_i=-2$ for
all $i$. Let $C'_z$ be the strict transform of $C$.
\begin{proposition}[see \expandafter{\cite[Lemma 4]{OZ}}]
We have $\omega(K)=K_z\cdot(K_z+C'_z)$.
\end{proposition}

We conclude this section with the following simple estimate for quasi-homogeneous singular points
\begin{proposition}\label{prop:simpleestimate}
If $z$ is a quasihomogeneous singular point, topologically equivalent to $x^p-y^q=0$ with $p$, $q$ coprime, then
\[\omega(K)<p+q.\]
Moreover, if $p=2$, then $\omega(z)=\frac{q-1}{2}$.
\end{proposition}
\begin{proof}
The second part follows from the fact that the multiplicity sequence for the singular point $x^2-y^{2k+1}=0$ is a length $k$ sequence $(2,\ldots,2)$.
For the first part we observe that the multiplicity sequence is constructed as follows. Suppose $p<q$. Then $m_1=p$. The blow up replaces $(p,q)$
by $(p,q-p)$ or $(q-p,p)$ depending on whether $p<q-p$ or $q-p<p$. It follows that $\sum_{i=2}^nm_i=q$, hence
\[\sum_{i=1}^n m_i=p+q.\]
Therefore $\omega(K)\le p+q-1$.
\end{proof}
The implication of Proposition~\ref{prop:simpleestimate} is that $\omega$ is a \emph{linear} invariant, that is, its value
for a $T_{pq}$-torus knot grows like the sum of $p$ and $q$, not like the product. The latter behavior is more typical, the genus and
the signature are examples.

Another inequality involving $\omega(z)$ is a generalization of the Orevkov--Zajdenberg inequality \cite[Section 11]{OZ}.
\begin{proposition}\label{prop:orevkovineq}
For a cuspidal singular point $z$ with Milnor number $\mu$ and multiplicity $m$ we have $\mu\le m\omega$.
\end{proposition}
\begin{proof}
Let $m_1,\ldots,m_n$ be the multiplicity sequence of $z$. We have $m_1\ge m_2\ge\ldots\ge m_n$. This means that $m_1\sum(m_i-1)\ge \sum m_i(m_i-1)$,
but in light of \eqref{eq:milnor}, this is precisely the statement of the proposition.
\end{proof}

\section{The family $K_n$}\label{sec:kn}

For $n\ge 1$, our knot $K_n$ is given by the surgery description shown in Figure \ref{fig:knot_0}.
Let $L=K\cup C_1\cup C_2$ be the oriented link as shown in Figure \ref{fig:knot_0}.
If we perform $(-1)$-surgery on $C_1$ and $(-\frac{1}{n+1})$-surgery on $C_2$, then
$K$ is changed into $K_n$.
Thus $K_n$ is the closure of the $4$-braid
\[
[2,1,3,2, (3,2,1)^4, 3^{2(n+1)},2],
\]
where an integer $k$ denotes the standard braid generator $\sigma_k$ of the $4$-string braid group.
In particular, $K_1$ is $m211$, and $K_2$ is  $t09284$  in the SnapPy census \cite{CDGsnappy}.
Since $K_n$ is the closure of a positive braid, $K_n$ is fibered and its genus is equal to $n+8$.
The diagram of $K_1$ is given in  Figure~\ref{fig:knot_0}. 

\begin{figure}[h]
\begin{minipage}{5cm}
\includegraphics*[scale=0.7]{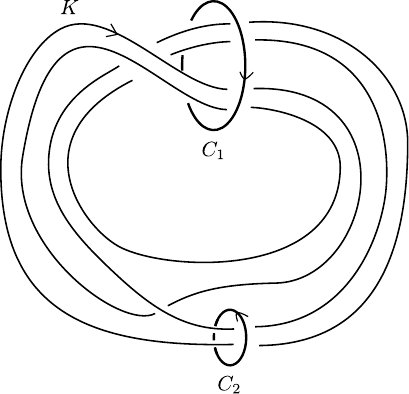}
\end{minipage}\hskip2cm
  \begin{minipage}{5cm}
 \includegraphics[width=4.5cm]{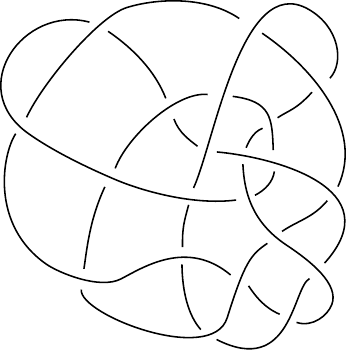}
\end{minipage}\hskip 0.5cm
\caption{\emph{Left:} The surgery description of $K_n$.
For the link $L=K\cup C_1\cup C_2$, perform $(-1)$-surgery on $C_1$ and $(-\frac{1}{n+1})$-surgery on $C_2$.
Then $K$ is changed to $K_n$. \emph{Right:} The knot $K_1=m211$.}\label{fig:knot_0}
\end{figure}

\begin{theorem}\label{thm:integral_n}
For $n\ge 1$, the knot $K_n$ enjoys the following.
\begin{enumerate}
\item
 $K_n$ is hyperbolic.
\item
$(4n+24)$-surgery on $K_n$ gives an L-space, so $K_n$ is an L-space knot.
\item
For the Upsilon invariant $\Upsilon_{K_n}(t)$ of $K_n$,
\[
I=\int_0^2 \Upsilon_{K_n}(t)\, dt= -\left(n+\frac{34}{5} \right).
\]
Thus $-3I=3n+\frac{102}{5}\not\in \mathbb{Z}$.
\end{enumerate}
\end{theorem}
The proof of the second property is deferred to Section~\ref{sec:montesinos}. 
Our first step towards proving Theorem~\ref{thm:integral_n} is to compute the Alexander polynomial. Once we know it, using
the fact that $K_n$ are L-space knots, we compute the integral proving item~(3). Moreover, knowing the Alexander polynomial, will allow us
to prove that $K_n$ is not hyperbolic.
\begin{lemma}\label{lem:alex}
The Alexander polynomial of $K_n$ is given as
\[
\begin{split}
\Delta_{K_n}(t)&=( t^{2n+16}-t^{2n+15}) + (t^{2n+12}-t^{2n+11} )+ (t^{2n+10}-t^{2n+9} )\\
 &\quad +(t^{2n+7}-t^{2n+6})+(t^{2n+5}-t^{2n+4})+\dots +(t^{11}-t^{10})+t^9  \\
& \quad  -t^7+t^6-t^5+t^4-t+1.
\end{split}
\]

\end{lemma}

\begin{proof}
  Let $L$ be the link as in Figure~\ref{fig:knot_0}.
  The multivariable Alexander polynomial of $L$ can be readily computed using \cite{CDGsnappy,Kodama}:
\[
\begin{split}
\Delta_L(x,y,z) &=x^7y^3z-x^5y^3z+x^5y^2z+x^5y^2-x^4y^2-x^5y\\
&\quad -2x^3y^2z+2x^4y+x^2y^2z+x^3yz-x^2yz-x^2y+x^2-1,
\end{split}
\]
where the variables $x,y,z$ correspond to the (oriented) meridians of $K$, $C_1$, $C_2$, respectively.

If we perform $(-1)$-surgery on $C_1$ and $(-\frac{1}{n+1})$-surgery on $C_2$, then
the link $K\cup C_1\cup C_2$ is changed into $K_n\cup C_1^n\cup C_2^n$.
Since these links have homeomorphic exteriors,
the induced isomorphism on the homology groups 
relates their Alexander polynomials; compare \cite{F1954,M2006}.

Let $\mu_{K}$, $\mu_{C_1}$ and $\mu_{C_2}$ be the homology classes of meridians
of $K$, $C_1$, $C_2$, respectively.
Each meridian has linking number one with the corresponding component.
Furthermore, let  $\lambda_{C_1}$ and $\lambda_{C_2}$ be the homology classes of their oriented longitudes.
We see that $\lambda_{C_1}=4\mu_{K}$ and $\lambda_{C_2}=2 \mu_{K}$.

Next, 
let $\mu_{K_n}$, $\mu_{C_1^n}$ and $\mu_{C_2^n}$ be the homology classes of meridians 
of $K_n$, $C_1^n$ and $C_2^n$.
Then we have that
$\mu_{K_n}=\mu_{K}$, $\mu_{C_1^n}=-\mu_{C_1}+ \lambda_{C_1}$, $\mu_{C_2^n}=-\mu_{C_2}+(n+1)\lambda_{C_2}$.
Hence
\[
\mu_{K}=\mu_{K_n}, \quad \mu_{C_1}=-\mu_{C_1^n}+4\mu_{K_n}, \quad \mu_{C_2}=-\mu_{C_2^n}+2(n+1)\mu_{K_n}.
\]
Thus, we have the relation between the Alexander polynomials as
\begin{equation}\label{eq:relation-alexander}
\Delta_{K_n\cup C_1^n\cup C_2^n}(x,y,z)=\Delta_{L}(x,x^{4}y^{-1},x^{2(n+1)}z^{-1} ).
\end{equation}

Since $\mathrm{lk}(K_n,C_2^n)=\mathrm{lk}(K,C_2)=2$ and $\mathrm{lk}(C_1^n,C_2^n)=\mathrm{lk}(C_1,C_2)=0$,
 the Torres condition \cite{T1953} gives
\begin{align*}
\Delta_{K_n\cup C_1^n\cup C_2^n}(x,y,1) &= (x^2y^0-1)\Delta_{K_n\cup C_1^n}(x,y)\\
&= (x^2-1)\Delta_{K_n\cup C_1^n}(x,y).
\end{align*}

Furthermore, since $\mathrm{lk}(K_n,C_1^n)=\mathrm{lk}(K,C_1)=4$, 
\[
\Delta_{K_n\cup C_1^n}(x,1)=\frac{x^4-1}{x-1}\Delta_{K_n}(x).
\]
Thus
\[
\Delta_{K_n}(x)=\frac{x-1}{x^4-1}\Delta_{K_n\cup C_1^n}(x,1)=\frac{x-1}{(x^4-1)(x^2-1)}\Delta_{K_n\cup C_1^n\cup C_2^n}(x,1,1).
\]
Then the relation (\ref{eq:relation-alexander}) gives
\[
\begin{split}
\Delta_{K_n}(t)&=\frac{t-1}{(t^4-1)(t^2-1)} \Delta_L(t,t^{4},t^{2(n+1)})\\
&=\frac{1}{(t^4-1)(t+1)} 
(t^{2n+21}-t^{2n+19}+t^{2n+15}-2t^{2n+13}+t^{2n+12}+t^{2n+9}-t^{2n+8}\\
&\quad + t^{13}-t^{12}-t^9+2t^8-t^6+t^2-1)\\
&=\frac{1}{(t^4-1)(t+1)} ( t^{2n+13}(t^8-1)-t^{2n+15}(t^4-1)-t^{2n+9}(t^4-1)\\
&\quad +t^{2n+8}(t^4-1)+t^9(t^4-1)-t^8(t^4-1)-t^2(t^4-1)+(t^8-1)           )\\
&=\frac{1}{t+1} ( t^{2n+13}(t^4+1)-t^{2n+15}
-t^{2n+9}+t^{2n+8}+t^9-t^8-t^2+(t^4+1) )\\
&= \frac{1}{t+1} ( t^{2n+15}(t^2-1) +t^{2n+9}(t^4-1)+t^9(t^{2n-1}+1)-t^4(t^4-1)-(t^2-1))\\
&=t^{2n+15}(t-1)+t^{2n+9}(t^2+1)(t-1)+t^9 \frac{t^{2n-1}+1}{t+1}-t^4(t^2+1)(t-1)-(t-1) \\
&= t^{2n+16}-t^{2n+15} + t^{2n+12}-t^{2n+11}+ t^{2n+10}-t^{2n+9}  +t^9 \left( \sum_{i=1}^{n-1}(t^{2i}-t^{2i-1}) + 1\right) \\
& \quad  -t^7+t^6-t^5+t^4-t+1.
\end{split}
\]
\end{proof}

In Definition~\ref{def:formal} we have recalled the notion of a formal semigroup of an L-space knot. While the proof that $K_n$
are L-space knots is given only in Section~\ref{sec:montesinos}, it is convenient to determine the $\Upsilon$ function
of $K_n$ from the Alexander polynomial we just computed.

First, we show that $K_n$ is not algebraic.
\begin{lemma}\label{lem:formal}
The formal semigroup $S_{K_n}$ of $K_n$ is not closed under addition.
\end{lemma}

\begin{proof}
By Lemma \ref{lem:alex},
the formal semigroup $S_{K_n}$ of $K_n$ starts
\[
0,4,6,9,11,12,\dots, 2n+5,\dots.
\]
Thus $4\in S_{K_n}$, but $8\not\in S_{K_n}$.
\end{proof}

\begin{lemma}\label{lem:upsilon}
For $n\ge 1$, 
the Upsilon function of $K_n$ is given as
\[
\Upsilon_{K_n}(t)=
\begin{cases}
-(n+8)t & (0\le t \le \frac{1}{2})\\
-(n+4)t-2 & (\frac{1}{2}\le  t \le \frac{4}{5}) \\
-(n-1)t-6 & (\frac{4}{5} \le t \le 1).
\end{cases}
\]
For $t\in [1,2]$, $\Upsilon_{K_n}(t)=\Upsilon_{K_n}(2-t)$.
\end{lemma}

Figure \ref{fig:upsilon} shows $\Upsilon_{K_n}(t)$ when $n=1$ and $n\ge 2$.

\begin{figure}[H]
\includegraphics*[scale=0.8]{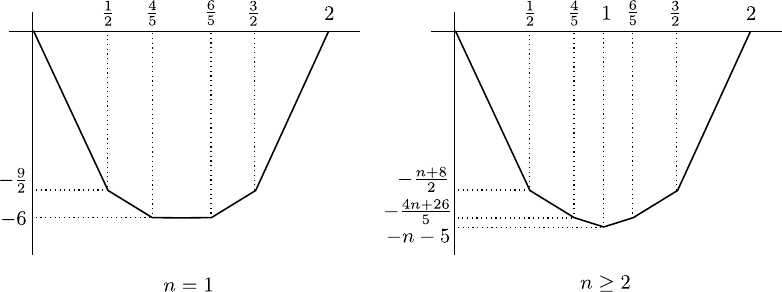}
\caption{The Upsilon function $\Upsilon_{K_n}(t)$.  Left shows the case $n=1$, and Right shows the case $n\ge 2$. }
\label{fig:upsilon}
\end{figure}

\begin{proof}
By \cite{BorodzikHedden}, the Upsilon function is the Fenchel--Legendre transform of the gap function $G(x)=2J(-x)$, in their notation,  determined by the Alexander polynomial.
In fact, there is a handy description of the graph of $G(x)$ as follows.
Let us write $\Delta_{K_n}(t)$ as $t^{a_0}-t^{a_1}+t^{a_2}-\dots+t^{a_{2n}}$.
Then the sequence of the jumps in the exponents is given as
\begin{equation}\label{eq:gap}
a_1-a_0,a_2-a_1,\dots, a_{2n}-a_{2n-1}.
\end{equation}

Consider the vectors $\boldsymbol{u}=(1,2)$ and $\boldsymbol{h}=(1,0)$ on $\mathbb{R}^2$.
Then \cite[Lemma 2.2]{Te2024} shows that
the graph of $G(x)$ restricted on $[-g, g]$ has a form of staircase specified by (\ref{eq:gap}). 
More precisely, we start at the point $(-g, 0)$, and move along $\boldsymbol{u}$ $a_1 -a_0$ times, 
along  $\boldsymbol{h}$ $a_2 -a_1$ times, and so on. Finally, we reach the point  $(g, 2g)$.
The function $G(x)$ is $0$ for $x\le -g$, and $2x$ for $x\ge g$.

By Lemma \ref{lem:alex},
the sequence of the jumps is
\[
1,3,1,1,1,2, \underbrace{1,1, \dots, 1,1}_{2n-2},2,1,1,1,3,1.
\]
Thus the graph of $G(x)$ goes through the points
$(-g,0)=(-n-8,0)$, $(-n-4,2)$, $(-n+1,6)$, $(n-1,2n+4)$, $(n+4,2n+10)$ and $(g,2g)=(n+8,2n+16)$,
which determine the convex hull.
More precisely, the convex hull of $G(x)$ is determined by the lines

\[
\begin{cases}
y=0 & (x\le -n-8) \\
y=\frac{x+n+8}{2} & (-n-8\le x\le -n-4)\\
y=\frac{4}{5}(x+n+4)+2 & ( -n-4\le x \le -n+1) \\
y=x+n+5 & (-n+1\le x \le n-1) \\
y=\frac{6(x-n+1)}{5}+2n+4 & (n-1\le x \le n+4) \\
y=\frac{3(x-n-4)}{2}+2n+10 & (n+4 \le x \le n+8) \\
y=2x & (x\ge n+8).
\end{cases}
\]
See Figure~\ref{fig:m111} for the case $n=1$.
Then the Fenchel--Legendre transformation immediately gives the conclusion.
\end{proof}
\begin{figure}[H]
\begin{tikzpicture}
\begin{scope}[xshift = -5.2cm,scale=0.8]
\draw[thin,black!50,dashed] (-0.5,0.0) -- (7,0.0);
\draw[thin,black!50,dashed] (0.0,-0.5) -- (0.0,7);
\draw (-0.2,-0.2) node [scale = 0.8] {$0$};
\draw (-0.2,-0.2) node [scale = 0.8] {$0$};
\draw[thin,black!50,dashed] (-0.5,0.7777777777777778) -- (7,0.7777777777777778);
\draw[thin,black!50,dashed] (0.7777777777777778,-0.5) -- (0.7777777777777778,7);
\draw (0.5777777777777777,-0.2) node [scale = 0.8] {$1$};
\draw (-0.2,0.5777777777777777) node [scale = 0.8] {$1$};
\draw[thin,black!50,dashed] (-0.5,1.5555555555555556) -- (7,1.5555555555555556);
\draw[thin,black!50,dashed] (1.5555555555555556,-0.5) -- (1.5555555555555556,7);
\draw (1.3555555555555556,-0.2) node [scale = 0.8] {$2$};
\draw (-0.2,1.3555555555555556) node [scale = 0.8] {$2$};
\draw[thin,black!50,dashed] (-0.5,2.3333333333333335) -- (7,2.3333333333333335);
\draw[thin,black!50,dashed] (2.3333333333333335,-0.5) -- (2.3333333333333335,7);
\draw (2.1333333333333333,-0.2) node [scale = 0.8] {$3$};
\draw (-0.2,2.1333333333333333) node [scale = 0.8] {$3$};
\draw[thin,black!50,dashed] (-0.5,3.111111111111111) -- (7,3.111111111111111);
\draw[thin,black!50,dashed] (3.111111111111111,-0.5) -- (3.111111111111111,7);
\draw (2.911111111111111,-0.2) node [scale = 0.8] {$4$};
\draw (-0.2,2.911111111111111) node [scale = 0.8] {$4$};
\draw[thin,black!50,dashed] (-0.5,3.888888888888889) -- (7,3.888888888888889);
\draw[thin,black!50,dashed] (3.888888888888889,-0.5) -- (3.888888888888889,7);
\draw (3.6888888888888887,-0.2) node [scale = 0.8] {$5$};
\draw (-0.2,3.6888888888888887) node [scale = 0.8] {$5$};
\draw[thin,black!50,dashed] (-0.5,4.666666666666667) -- (7,4.666666666666667);
\draw[thin,black!50,dashed] (4.666666666666667,-0.5) -- (4.666666666666667,7);
\draw (4.466666666666667,-0.2) node [scale = 0.8] {$6$};
\draw (-0.2,4.466666666666667) node [scale = 0.8] {$6$};
\draw[thin,black!50,dashed] (-0.5,5.444444444444445) -- (7,5.444444444444445);
\draw[thin,black!50,dashed] (5.444444444444445,-0.5) -- (5.444444444444445,7);
\draw (5.2444444444444445,-0.2) node [scale = 0.8] {$7$};
\draw (-0.2,5.2444444444444445) node [scale = 0.8] {$7$};
\draw[thin,black!50,dashed] (-0.5,6.222222222222222) -- (7,6.222222222222222);
\draw[thin,black!50,dashed] (6.222222222222222,-0.5) -- (6.222222222222222,7);
\draw (6.022222222222222,-0.2) node [scale = 0.8] {$8$};
\draw (-0.2,6.022222222222222) node [scale = 0.8] {$8$};
\draw[->,thick] (-0.5,0) -- (7,0);
\draw[->,thick] (0,-0.5) -- (0,7);
\draw[blue!50!black,thick] (0.0,7.0) -- (0.7777777777777778,7.0);
\draw[blue!50!black,thick] (0.7777777777777778,7.0) -- (0.7777777777777778,6.222222222222222);
\draw[blue!50!black,thick] (0.7777777777777778,6.222222222222222) -- (0.7777777777777778,5.444444444444445);
\draw[blue!50!black,thick] (0.7777777777777778,5.444444444444445) -- (0.7777777777777778,4.666666666666667);
\draw[blue!50!black,thick] (0.7777777777777778,4.666666666666667) -- (1.5555555555555556,4.666666666666667);
\draw[blue!50!black,thick] (1.5555555555555556,4.666666666666667) -- (1.5555555555555556,3.888888888888889);
\draw[blue!50!black,thick] (1.5555555555555556,3.888888888888889) -- (2.3333333333333335,3.888888888888889);
\draw[blue!50!black,thick] (2.3333333333333335,3.888888888888889) -- (2.3333333333333335,3.111111111111111);
\draw[blue!50!black,thick] (2.3333333333333335,3.111111111111111) -- (2.3333333333333335,2.3333333333333335);
\draw[blue!50!black,thick] (2.3333333333333335,2.3333333333333335) -- (3.111111111111111,2.3333333333333335);
\draw[blue!50!black,thick] (3.111111111111111,2.3333333333333335) -- (3.888888888888889,2.3333333333333335);
\draw[blue!50!black,thick] (3.888888888888889,2.3333333333333335) -- (3.888888888888889,1.5555555555555556);
\draw[blue!50!black,thick] (3.888888888888889,1.5555555555555556) -- (4.666666666666667,1.5555555555555556);
\draw[blue!50!black,thick] (4.666666666666667,1.5555555555555556) -- (4.666666666666667,0.7777777777777778);
\draw[blue!50!black,thick] (4.666666666666667,0.7777777777777778) -- (5.444444444444445,0.7777777777777778);
\draw[blue!50!black,thick] (5.444444444444445,0.7777777777777778) -- (6.222222222222222,0.7777777777777778);
\draw[blue!50!black,thick] (6.222222222222222,0.7777777777777778) -- (7.0,0.7777777777777778);
\draw[blue!50!black,thick] (7.0,0.7777777777777778) -- (7.0,0.0);
\draw[red!50!black,thick] (0.0,7.0) -- (0.7777777777777778,4.666666666666667);
\draw[red!50!black] (0.1888888888888889,5.533333333333334) node [scale=0.8] {$-\frac{3}{1}$};
\draw[red!50!black,thick] (0.7777777777777778,4.666666666666667) -- (2.3333333333333335,2.3333333333333335);
\draw[red!50!black] (1.2555555555555555,3.2) node [scale=0.8] {$-\frac{3}{2}$};
\draw[red!50!black,thick] (2.3333333333333335,2.3333333333333335) -- (4.666666666666667,0.7777777777777778);
\draw[red!50!black] (3.2,1.2555555555555555) node [scale=0.8] {$-\frac{2}{3}$};
\draw[red!50!black,thick] (4.666666666666667,0.7777777777777778) -- (7.0,0.0);
\draw[red!50!black] (5.233333333333334,0.2888888888888889) node [scale=0.8] {$-\frac{1}{3}$};
\fill[red!30!black,thick] (0.0,7.0) circle (0.08);
\draw(-0.4,6.8) node [scale=0.8] {(0,9)};
\fill[red!30!black,thick] (0.7777777777777778,4.666666666666667) circle (0.08);
\draw(0.6777777777777778,4.266666666666667) node [scale=0.8] {(1,6)};
\fill[red!30!black,thick] (2.3333333333333335,2.3333333333333335) circle (0.08);
\draw(2.2333333333333334,2.0333333333333333) node [scale=0.8] {(3,3)};
\fill[red!30!black,thick] (4.666666666666667,0.7777777777777778) circle (0.08);
\draw(4.366666666666667,0.5777777777777777) node [scale=0.8] {(6,1)};
\fill[red!30!black,thick] (7.0,0.0) circle (0.08);
\draw(6.7,-0.3) node [scale=0.8] {(9,0)};
\fill[blue!30!black,thick] (0.0,7.0) circle (0.08);
\fill[blue!30!black,thick] (0.7777777777777778,7.0) circle (0.08);
\fill[blue!30!black,thick] (0.7777777777777778,4.666666666666667) circle (0.08);
\fill[blue!30!black,thick] (1.5555555555555556,4.666666666666667) circle (0.08);
\fill[blue!30!black,thick] (1.5555555555555556,3.888888888888889) circle (0.08);
\fill[blue!30!black,thick] (2.3333333333333335,3.888888888888889) circle (0.08);
\fill[blue!30!black,thick] (2.3333333333333335,2.3333333333333335) circle (0.08);
\fill[blue!30!black,thick] (3.888888888888889,2.3333333333333335) circle (0.08);
\fill[blue!30!black,thick] (3.888888888888889,1.5555555555555556) circle (0.08);
\fill[blue!30!black,thick] (4.666666666666667,1.5555555555555556) circle (0.08);
\fill[blue!30!black,thick] (4.666666666666667,0.7777777777777778) circle (0.08);
\fill[blue!30!black,thick] (7.0,0.7777777777777778) circle (0.08);
\fill[blue!30!black,thick] (7.0,0.0) circle (0.08);
\end{scope}
\begin{scope}[xshift=4.5cm, yshift = -0.05cm, scale = 0.28]
  \clip (-13,-1) rectangle (11,21);
  \draw[dashed, black!30!white] (-13,0) grid (11,20);
  \draw[->] (-13,0) -- (11,0);
  \draw[->] (0,-1) -- (0,20);
  \coordinate(A0) at (-13,0);
  \coordinate(A1) at (-9,0);
  \coordinate(A2) at (-8,2);
  \coordinate(A3) at (-5,2);
  \coordinate(A4) at (-4,4);
  \coordinate(A5) at (-3,4);
  \coordinate(A6) at (-2,6);
  \coordinate(A7) at (0,6);
  \coordinate(A8) at (2,10);
  \coordinate(A9) at (3,10);
  \coordinate(A10) at (4,12);
  \coordinate(A11) at (5,12);
  \coordinate(A12) at (8,18);
  \coordinate(A13) at (9,18);
  \coordinate(A14) at (11,22);
  \draw[ultra thick, blue!30!black] (A0) -- (A1) -- (A2) -- (A3) -- (A4) -- (A5) -- (A6) -- (A7) -- (A8) -- (A9) -- (A10) -- (A11) -- (A12) -- (A13) -- (A14);
  \draw[thick,red] (A1) -- (A3) -- (A7) -- (A11) -- (A13);
  \fill[blue!70!black] (A1) circle (0.20);
  \fill[blue!70!black] (A2) circle (0.20);
  \fill[blue!70!black] (A3) circle (0.20);
  \fill[blue!70!black] (A4) circle (0.20);
  \fill[blue!70!black] (A5) circle (0.20);
  \fill[blue!70!black] (A6) circle (0.20);
  \fill[blue!70!black] (A7) circle (0.20);
  \fill[blue!70!black] (A8) circle (0.20);
  \fill[blue!70!black] (A9) circle (0.20);
  \fill[blue!70!black] (A10) circle (0.20);
  \fill[blue!70!black] (A11) circle (0.20);
  \fill[blue!70!black] (A12) circle (0.20);
  \fill[blue!70!black] (A13) circle (0.20);
  \draw(5,-0.2) -- (5,0.2); \draw[thick] (5,-0.6) node [scale=0.8] {$5$};
  \draw(-5,-0.2) -- (-5,0.2); \draw[thick] (-5,-0.6) node [scale=0.8] {$-5$};
  \draw(10,-0.2) -- (10,0.2); \draw[thick] (10,-0.6) node [scale=0.8] {$10$};
  \draw(-10,-0.2) -- (-10,0.2); \draw[thick] (-10,-0.6) node [scale=0.8] {$-10$};
  \draw(-0.2,5) -- (0.2,5); \draw (0.6,5) node [scale=0.9] {$5$};
  \draw(-0.2,10) -- (0.2,10); \draw (0.6,10) node [scale=0.9] {$10$};
  \draw(-0.2,15) -- (0.2,15); \draw (0.6,15) node [scale=0.9] {$15$};
\end{scope}

\end{tikzpicture}    
\caption{The knot $K_1$.
Left: the staircase. Right: the graph of the gap function. The red lines are piecewise-linear convex supporting functions, used to compute
the Fenchel--Legendre transform.}\label{fig:m111}
\end{figure}
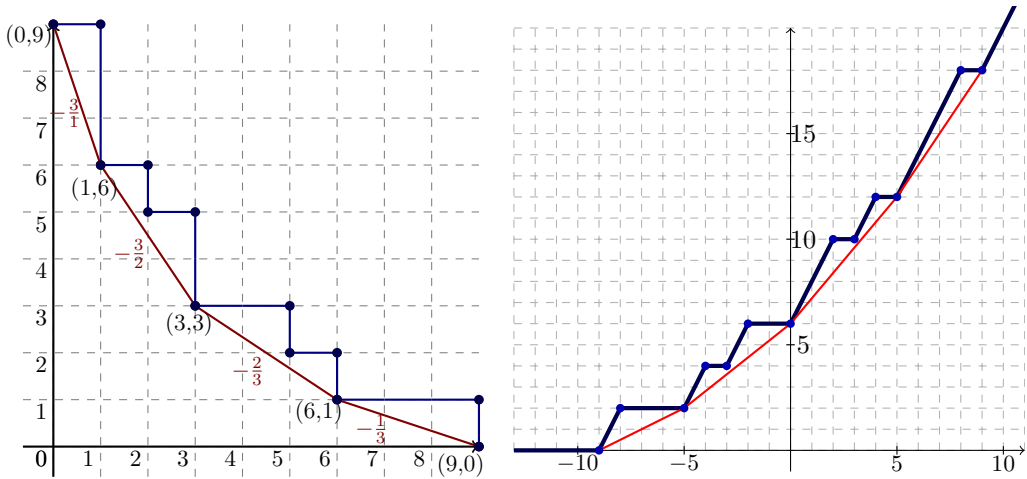


In the next section, we will prove that $(4n+24)$-surgery on $K_n$ yields an L-space by using the Montesinos trick.
Once we admit it, it is easy to prove that $K_n$ is hyperbolic.

\begin{lemma}\label{lem:hyperbolic}
$K_n$ is hyperbolic.
\end{lemma}

\begin{proof}
Recall that a torus knot of type $(p,q)$ $(0<p<q)$ has
the formal semigroup $\langle p,q \rangle=\{ap+bq \mid a, b \ge 0\}$, which is closed under addition.
Hence $K_n$ is not a torus knot by Lemma \ref{lem:formal}.

By \cite{Kr2015}, $K_n$ is prime.
Assume that $K_n$ is a satellite knot for a contradiction.
Since the bridge number of  $K_n$ is at most four, 
the companion is a $2$-bridge knot and the pattern has wrapping number two \cite{Sc1954}.
Since $K_n$ is an L-space knot, the companion and the pattern knot are also L-space knots \cite{BM2019,Ho2016}.
Furthermore,  the pattern is braided.
Thus the companion is a $2$-bridge torus knot by \cite{OSlens}, and $K_n$ is its $2$-cable.
In other words, $K_n$ is an iterated torus L-space knot.
Finally, Wang \cite{Wang} shows that the formal semigroup of such a knot is closed under addition.
This contradicts Lemma \ref{lem:formal}.
\end{proof}

\begin{proof}[Proof of Theorem \ref{thm:integral_n}]
This immediately follows  from
Lemmas \ref{lem:upsilon} and~\ref{lem:hyperbolic}, and Proposition \ref{prop:l-space}.
\end{proof}

\section{Montesinos trick}\label{sec:montesinos}

In this section, we prove that the $(4n+24)$-surgery on $K_n$ yields an L-space by using the Montesinos trick \cite{Mon1975}.
For a surgery diagram of a strongly invertible knot or link, the Montesinos trick describes the resulting closed $3$-manifold as the
double branched cover of another knot or link obtained from the tangle replacements corresponding to the surgery coefficients.

Figure \ref{fig:invertible} shows 
a surgery diagram in a strongly invertible position with the axis $A$.
After performing $(-1)$-surgery, we have $K_n$ with surgery coefficient $4n+24$.

\begin{figure}
\includegraphics*[scale=0.5]{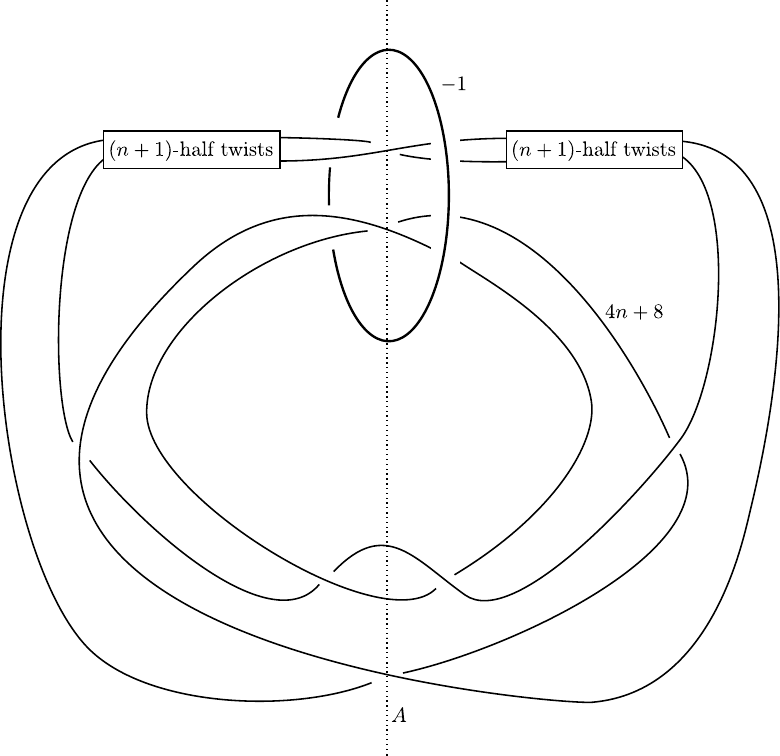}
\caption{
After $(-1)$-surgery, the surgery diagram gives $(4n+24)$-surgery on $K_n$.
Each rectangle box contains right-handed $(n+1)$-half twists.}
\label{fig:invertible}
\end{figure}

Take the quotient of $K_n\cup A$ under the involution along $A$.
Then we obtain a two-component link shown in Figure \ref{fig:mont1}.
The Montesinos trick claims that the resulting manifold of $(4n+24)$-surgery on $K_n$ is 
the double branched cover of $S^3$ branched over this link.
Figures \ref{fig:mont2} and \ref{fig:mont456} show the deformations of the link.

\begin{figure}
\includegraphics*[scale=0.5]{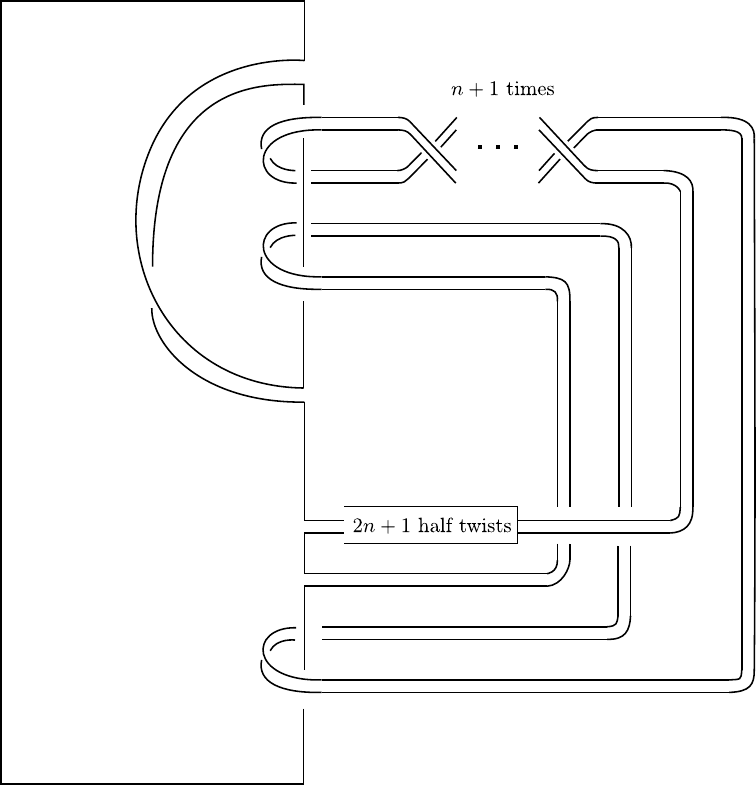}
\caption{The double branched cover of $S^3$ branched over this link gives the resulting manifold of
$(4n+24)$-surgery on $K_n$.
}
\label{fig:mont1}
\end{figure}

\begin{figure}[H]
\includegraphics*[scale=0.45]{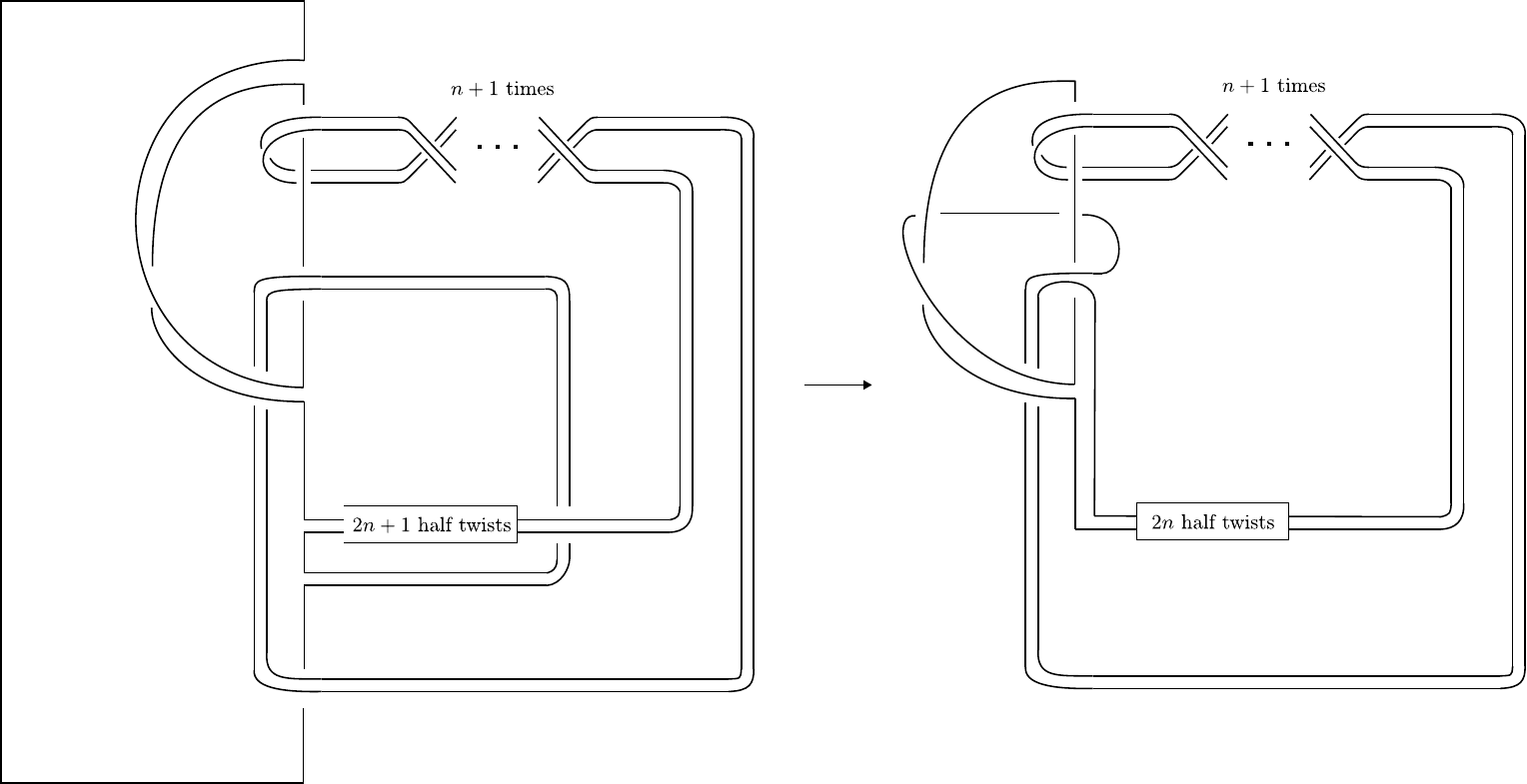}
\caption{Deformation of the link.}
\label{fig:mont2}
\end{figure}

\begin{figure}[H]
\includegraphics*[scale=0.45]{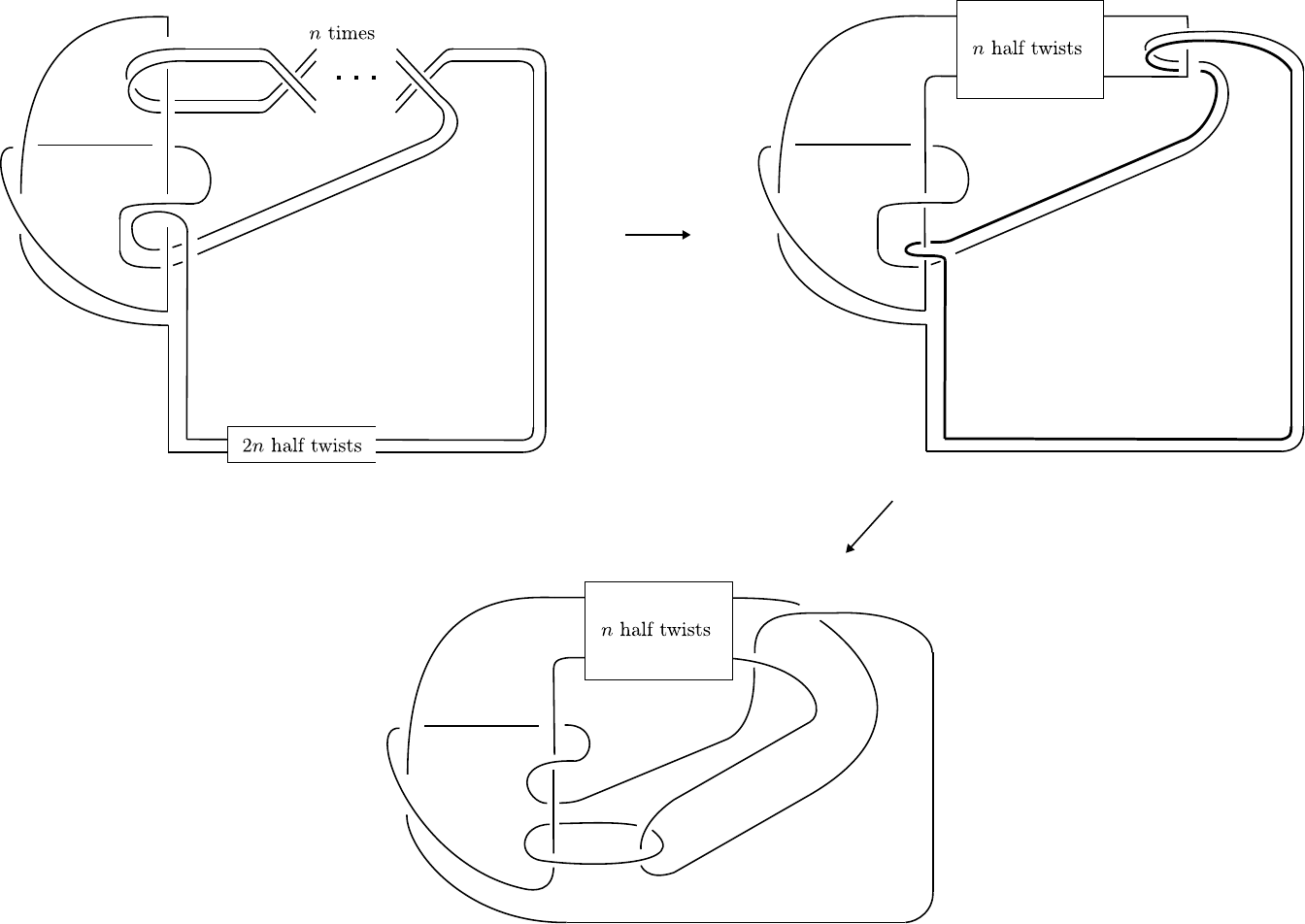}
\caption{Deformation of the link (continued).}
\label{fig:mont456}
\end{figure}

Let us denote this link by $\ell_n$.
We perform two resolutions as shown in Figure \ref{fig:resolution}
at a crossing located in the box with $n$ half twists.
Let $\ell_\infty$ and $\ell_{n-1}$ be the resulting links.
Then it is straightforward to calculate $\det \ell_n=4n+24$, $\det \ell_\infty=4$.
(For example, use the checkerboard colorings of the diagrams as shown in Figures \ref{fig:mont456} and \ref{fig:mont7}.)
This shows the equation $\det \ell_n=\det \ell_{n-1} + \det \ell_\infty$.
Thus if the double branched covers of $\ell_{n-1}$ and $\ell_\infty$ are L-spaces, then so is the double branched cover of $\ell_n$ \cite{BGW2013, OSlens, OSdouble2005}.

\begin{figure}[H]
\includegraphics*[scale=1]{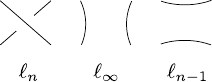}
\caption{Two resolutions.}
\label{fig:resolution}
\end{figure}

As shown in Figure \ref{fig:mont7},  the link $\ell_\infty$ is the Montesinos link $M(-1/2,-1/2,1/2)$.
Thus the double branched cover of $\ell_\infty$ is a Seifert fibered space over the $2$-sphere with three exceptional fibers of indices $2,2,2$,
which is elliptic.  Then it is an L-space \cite{OSlens}.

\begin{figure}[H]
\includegraphics*[scale=0.45]{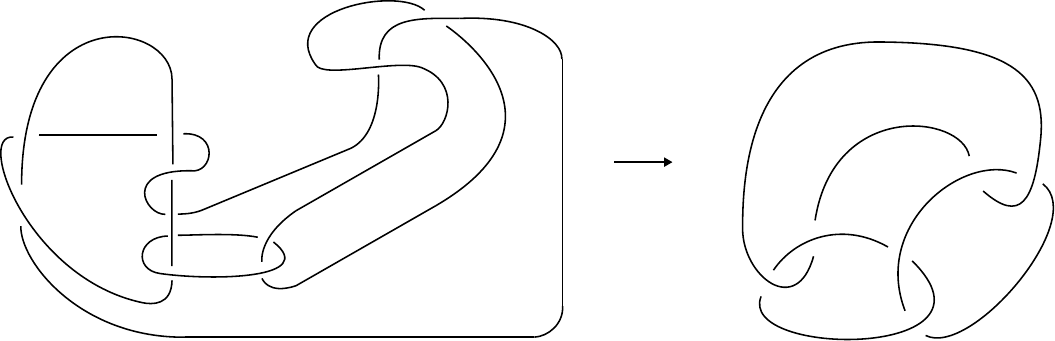}
\caption{The link $\ell_\infty$ with $\det \ell_\infty=4$ is the Montesinos link $M(-1/2,-1/2,1/2)$.
The double branched cover over this link is a Seifert fibered space over the $2$-sphere with three
exceptional fibers of indices $2,2,2$, which is elliptic.}
\label{fig:mont7}
\end{figure}

Inductively, it is enough to show that the double branched cover of $\ell_1$ is an L-space.
However, as shown in Figure \ref{fig:ell1},  $\ell_1$ is  the Montesinos link $M(1/2,1/2,-4/11)$.
The double branched cover is also an elliptic Seifert fibered space, which is an L-space.

\begin{figure}[H]
\includegraphics*[scale=0.45]{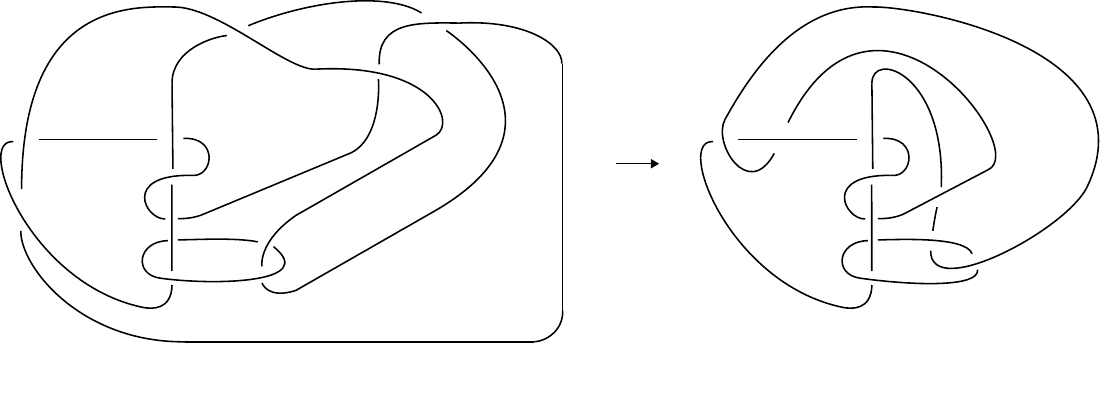}
\caption{The link $\ell_1$ is the Montesinos link $M(1/2,1/2,-4/11)$.}
\label{fig:ell1}
\end{figure}

Thus, we have shown the following.

\begin{proposition}\label{prop:l-space}
For $K_n$, $(4n+24)$-surgery yields an L-space.
\end{proposition}

\section{Linear independence of $K_n$}\label{sec:linear}
In this section we prove the following result.
\begin{theorem}\label{thm:independence}
  There is an increasing sequence $a_n$ such that the knots $K_{a_1},K_{a_2},\dots$ are linearly independent
  in the topological concordance group.
\end{theorem}
\begin{proof}
  Recall that for any knot $K$, the Tristram--Levine signature function $\sigma_K$ is a piecewise constant function from $S^1$ to $\Z$
  with discontinuities only at the roots of the Alexander polynomial. 
  We have the following classical result; see \cite[Chapter 12]{Kawauchi} for a classical approach and \cite{Powell_top} for the proof
  under topological concordance.
  \begin{proposition}
    If $K_1$ and $K_2$ are topologically concordant, then $\sigma_{K_1}(t)=\sigma_{K_2}(t)$ for all but finitely many $t$ in $S^1$.
  \end{proposition}
  The next result is a consequence of the definition of the signature, it can be deduced from the d\'evissage of the Blanchfield form. We refer
  an interested reader to \cite[Section 5]{BCP}.
  \begin{lemma}
    Suppose $\Delta_K(t)$ has a zero at $t_0\in S^1$ of odd multiplicity. Then, the signature function has jump at $t_0$, that is,
    \[\left|\lim_{u\to 0^+} \left(\sigma(t_0 e^{i u})-\sigma(t_0 e^{-iu})\right)\right|\ge 2.\]
  \end{lemma}
  Our goal is to find, for all knots $K_n$ with $n\ge 11$, a zero of $\Delta$ of odd multiplicity of $\Delta$ on $S^1$,
  and very close to $1$. First, we symmetrize the Alexander polynomial for $K_n$ computed in Lemma~\ref{lem:alex} above:

\begin{align*}
  \psi_n=&t^{-n-8}\left(( t^{2n+16}-t^{2n+15}) + (t^{2n+12}-t^{2n+11} )+ (t^{2n+10}-t^{2n+9} )\right.\\
 &\quad +(t^{2n+7}-t^{2n+6})+(t^{2n+5}-t^{2n+4})+\dots +(t^{11}-t^{10})+t^9  \\
& \quad  \left.-t^7+t^6-t^5+t^4-t+1\right).
\end{align*}
Decompose $\psi_n=\psi_n^1+\psi_n^2$, where
\begin{multline*}
  \psi^1_n=t^{n+8}+t^{-n-8}-(t^{n+7}+t^{-n-7})+t^{n+4}+t^{-n-4} - (t^{n+3}+t^{-n-3})+\\
  (t^{n+2}+t^{-n-2})-(t^{n+1}+t^{-n-1})+
  (t^{n-1}+t^{-n+1})-(t^{n-2}+t^{-n+2})
\end{multline*}
and
\[
  \psi^2_n=t^{-n-8}(t^{11}-t^{12}+\dots-t^{2n+4}+t^{2n+5}).
\]
It is now convenient to consider both functions on $[0,2\pi]$. To this end, write $t=e^{iu}$, and set
\[\alpha_n(u)=\frac12\psi^1_n(e^{iu}),\ \beta_n(u)=\frac12\psi^2_n(e^{iu}).\]
We first deal with $\alpha_n$. We have
  \begin{multline*}
  \alpha_n(u)=\cos(n+8)u-\cos(n+7)u+\cos(n+4)u-\cos(n+3)u+\cos(n+2)u-\\
\cos(n+1)u+\cos(n-1)u-\cos(n-2)u.
\end{multline*}
From the cosine sum formula we get
\[\alpha_n(u)=-2\sin\frac{u}{2}\left(\sin\frac{2n+15}{2}u+\sin\frac{2n+7}{2}u+\sin\frac{2n+3}{2}u+\sin\frac{2n-3}{2}u\right).\]
Applying the sine sum formula we obtain
\begin{equation}\label{eq:simplify}
\alpha_n(u)=-4\sin\frac{u}{2}\left(\sin(n+\frac{11}{2})u\cos 2u+\sin nu\cos \frac32u\right).
\end{equation}
The sign of the function $\alpha_n$ near $0$ can be examined using \eqref{eq:simplify}. Indeed, the sine function is positive on $(0,\pi)$, hence
$\sin(n+\frac{11}{2})u$ is positive on $(0,\frac{2\pi}{2n+11})$. On that interval, also $\cos 2u$, $\sin nu$ and $\cos\frac32u$ are
positive. That is:
\[\alpha_n(u)<0 \textrm{ for }u\in(0,\frac{2\pi}{2n+11}).\]
As for the expression $\psi^2_n$, we note that
\[\psi^2_n=t^{-n-8}\frac{t^{2n+6}+t^{11}}{t+1}=\frac{t^{n-5/2}+t^{-n+5/2}}{t^{1/2}+t^{-1/2}}.\]
Hence
\[\beta_n(u)=\frac{\cos(n-5/2)u}{2\cos u/2}.\]
Then, $\beta_n(0)=1/2$ and $\beta_n(\frac{\pi}{2n-5})=0$. We have assumed that $n\ge 11$. This implies that 
$\frac{\pi}{2n-5}\in(0,\frac{2\pi}{2n+11})$. 
Define $\gamma_n=\alpha_n+\beta_n$, so that $\gamma_n(u)=\frac12\psi_n(e^{iu})$. We have 
\[\gamma_n(0)=\frac12,\ \gamma_n(\frac{\pi}{2n-5})<0.\]
In other words, $\gamma_n$ changes sign on the interval $(0,\frac{\pi}{2n-5})$. Denote by $u_n$ the smallest positive zero of $\gamma_n$
of odd multiplicity. Note that $e^{iu_n}$ is a root of the Alexander polynomial.
\begin{remark}
  Computer experiments suggest that there is only one zero of $\gamma_n$ in that interval and that this zero is simple. However, we will
  not need this in the proof.
\end{remark}

We can now define an infinite increasing 
sequence $a_m$ such that $K_{a_m}$ are linearly independent. To this end set $a_1=11$. Suppose $a_1,\dots,a_m$
are already defined.
As the function $u\mapsto \gamma_{a_m}(u)$ is continuous, and $\gamma_{a_m}(0)=1$, there is $\lambda_m>0$ such that $\gamma_{a_m}(u)>0$ for all $u\in(0,\lambda_m)$. We choose $\lambda_m$ in such a way that $\lambda_m$ form a decreasing sequence of real positive numbers.

Positivity of $\gamma_{a_m}$ implies in particular that there are no jumps of the signature function of $K_{a_m}$,
that is the signature jumps at all values of $e^{iu}$ for $u\in[0,\lambda_m)$ are zero.

Choose $a_{m+1}$ by the condition that $\frac{\pi}{2a_{m+1}-5}<\lambda_m$. This means that $u_{a_{m+1}}\in (0,\lambda_m)$. 
Then, $\gamma_{a_{m+1}}$ vanishes on $u_{a_{m+1}}$, but for all $j\le m$, $\gamma_{a_m}$ is positive near $u_{a_{m+1}}$.
In particular, the signature jump at $e^{iu_{m+1}}$ for $K_{a_{m+1}}$ is not zero. That is, the map $\Psi_{m+1}$ assigning to a knot half its
signature jump at $e^{ia_{m+1}}$ is a homomorphism from the topological concordance group to $\Z$ that vanishes on the subgroup
spanned by $K_{a_1},\dots,K_{a_m}$, and is not zero on $K_{a_{m+1}}$.

The maps $\Psi_{1},\dots,$ define an isomorphism between the subgroup spanned by $K_{a_1},\dots,$ and $\Z^\infty$.
\end{proof}
We conclude the section by the following statement.
\begin{lemma}
  Let $\Upsilon'$ be the $\Upsilon$ function for the positive trefoil. Then, $\Upsilon_{K_{n+1}}=\Upsilon_{K_n}+\Upsilon'$.
\end{lemma}
\begin{proof}
  Follows immediately from Lemma~\ref{lem:upsilon}.
\end{proof}
\begin{corollary}
  The $\Upsilon$ function alone is not sufficient to show that $K_n$ are independent in the concordance group modulo the subgroup
  generated by the algebraic knots.
\end{corollary}
We remark that the signature jumps at $\zeta=e^{2\pi i/6}$ show that the trefoil does not belong to the subgroup spanned by all the $K_{a_i}$,
not even to the subgroup spanned by all the $K_n$. Note that Lemma~\ref{lem:alex} computes the Alexander polynomial of $K_n$
from the Alexander polynomial $\Delta_L$, which is common for all $n$. It follows that the value of the symmetrized Alexander polynomial of $K_n$,
$\psi_{n}(\zeta)$, depends only on $n\bmod 6$. For $n=1,\dots,6$, we can show that $\psi_n(\zeta)\neq 0$ by direct computations. Hence, none of the $\psi_n$ vanishes on $\zeta$. The signature jump at $\zeta$ is equal to zero for all the $K_n$, but it is not zero for the trefoil.

We do not continue this argument, because these methods alone are insufficient to prove independence of $K_n$ modulo the subgroup generated
by the algebraic knots. In fact, there exist linear combinations of algebraic knots with vanishing signature jumps, see \cite{HKL}.

\section{Specific knots}\label{sec:specific}
In \cite{BakerKegel}, Baker and Kegel shown a list of 632 hyperbolic L-space knots from the SnapPy census. For all of them, we have
computed the Alexander polynomial using SnapPy \cite{CDGsnappy}, and by a simple algorithm we have determined the $\Upsilon$ function.
The expression $-3\int\Upsilon$ turned out to be non-integral for 96 knots, with the denominators in the set
$\{3,5,7,10,11,14,15,21,30,35,70,105,385\}$.

\begin{table}[H]
\setlength{\extrarowheight}{1.5pt}
\begin{tabular}{|C|C|c|c|C|C|c|c|}\hline
  $m211$ &$\frac{117}{5}$ &$s560$ &$\frac{192}{5}$ &$v0319$ &$\frac{292}{5}$ &$v0545$ &$\frac{173}{5}$ \\[1mm]\hline 
$v0830$ &$\frac{237}{5}$ &$v1359$ &$\frac{1874}{35}$ &$v1423$ &$\frac{326}{7}$ &$v1565$ &$\frac{267}{5}$ \\[1mm]\hline 
$v2900$ &$\frac{389}{7}$ &$v3070$ &$\frac{445}{7}$ &$v3335$ &$\frac{188}{5}$ &$t00621$ &$\frac{467}{5}$ \\[1mm]\hline 
$t01966$ &$\frac{357}{5}$ &$t03106$ &$\frac{999}{14}$ &$t03710$ &$\frac{342}{5}$ &$t03843$ &$\frac{293}{5}$ \\[1mm]\hline 
$t04927$ &$\frac{571}{7}$ &$t06246$ &$\frac{3554}{35}$ &$t06637$ &$\frac{5987}{70}$ &$t06957$ &$\frac{2068}{21}$ \\[1mm]\hline 
$t08114$ &$\frac{148}{5}$ &$t08184$ &$\frac{3099}{35}$ &$t08936$ &$\frac{1979}{35}$ &$t09284$ &$\frac{132}{5}$ \\[1mm]\hline 
$t09633$ &$\frac{1566}{35}$ &$t09882$ &$\frac{108}{5}$ &$t10177$ &$\frac{690}{7}$ &$t11887$ &$\frac{308}{5}$ \\[1mm]\hline 
$t12288$ &$\frac{634}{7}$ &$t12533$ &$\frac{157}{5}$ &$o9\_01175$ &$\frac{642}{5}$ &$o9\_02383$ &$\frac{413}{5}$ \\[1mm]\hline 
$o9\_02909$ &$\frac{334}{7}$ &$o9\_04054$ &$\frac{348}{5}$ &$o9\_04060$ &$\frac{477}{5}$ &$o9\_07044$ &$\frac{4674}{35}$ \\[1mm]\hline 
$o9\_07152$ &$\frac{687}{5}$ &$o9\_07401$ &$\frac{767}{7}$ &$o9\_08402$ &$\frac{417}{5}$ &$o9\_09271$ &$\frac{233}{3}$ \\[1mm]\hline 
$o9\_09731$ &$\frac{867}{14}$ &$o9\_10192$ &$\frac{12346}{105}$ &$o9\_10213$ &$\frac{3727}{42}$ &$o9\_11556$ &$\frac{592}{5}$ \\[1mm]\hline 
$o9\_11658$ &$\frac{816}{7}$ &$o9\_12079$ &$\frac{787}{5}$ &$o9\_12253$ &$\frac{413}{5}$ &$o9\_12477$ &$\frac{1881}{14}$ \\[1mm]\hline 
$o9\_13054$ &$\frac{662}{7}$ &$o9\_16431$ &$\frac{5234}{35}$ &$o9\_17382$ &$\frac{637}{5}$ &$o9\_19247$ &$\frac{1067}{10}$ \\[1mm]\hline 
$o9\_19645$ &$\frac{517}{5}$ &$o9\_20029$ &$\frac{1082}{7}$ &$o9\_21620$ &$\frac{1671}{14}$ &$o9\_22252$ &$\frac{1138}{7}$ \\[1mm]\hline 
$o9\_23032$ &$\frac{1769}{35}$ &$o9\_23461$ &$\frac{6837}{70}$ &$o9\_23723$ &$\frac{4324}{35}$ &$o9\_24069$ &$\frac{9347}{70}$ \\[1mm]\hline 
$o9\_24126$ &$\frac{59548}{385}$ &$o9\_24407$ &$\frac{3293}{30}$ &$o9\_24946$ &$\frac{268}{5}$ &$o9\_25110$ &$\frac{2005}{21}$ \\[1mm]\hline 
$o9\_27371$ &$\frac{935}{7}$ &$o9\_27767$ &$\frac{15704}{105}$ &$o9\_28751$ &$\frac{3557}{30}$ &$o9\_29551$ &$\frac{1691}{15}$ \\[1mm]\hline 
$o9\_29648$ &$\frac{725}{7}$ &$o9\_30142$ &$\frac{1570}{11}$ &$o9\_31440$ &$\frac{503}{5}$ &$o9\_32065$ &$\frac{7047}{70}$ \\[1mm]\hline 
$o9\_32314$ &$\frac{1041}{14}$ &$o9\_33380$ &$\frac{252}{5}$ &$o9\_33430$ &$\frac{312}{5}$ &$o9\_33486$ &$\frac{3391}{21}$ \\[1mm]\hline 
$o9\_33801$ &$\frac{3204}{35}$ &$o9\_33959$ &$\frac{6477}{70}$ &$o9\_34689$ &$\frac{428}{5}$ &$o9\_35720$ &$\frac{363}{5}$ \\[1mm]\hline 
$o9\_36380$ &$\frac{1482}{11}$ &$o9\_36544$ &$\frac{2579}{35}$ &$o9\_37482$ &$\frac{849}{10}$ &$o9\_37551$ &$\frac{228}{5}$ \\[1mm]\hline 
$o9\_38287$ &$\frac{963}{14}$ &$o9\_38679$ &$\frac{879}{7}$ &$o9\_39162$ &$\frac{821}{14}$ &$o9\_39859$ &$\frac{2159}{35}$ \\[1mm]\hline 
$o9\_40026$ &$\frac{1656}{35}$ &$o9\_40363$ &$\frac{5427}{70}$ &$o9\_40487$ &$\frac{173}{5}$ &$o9\_42493$ &$\frac{2684}{35}$ \\[1mm]\hline 
$o9\_42675$ &$\frac{203}{5}$ &$o9\_42961$ &$\frac{355}{7}$ &$o9\_43750$ &$\frac{357}{5}$ &$o9\_43857$ &$\frac{2269}{35}$ \\[1mm]\hline 
\end{tabular}
\end{table}
\bibliographystyle{abbrv}
\def\MR#1{}
\bibliography{biblio.bib}

\end{document}